\documentclass[11pt]{article} 
\usepackage{a4,amssymb, amsmath, amsthm}

\newtheorem{thm}{Theorem}[section]
\newtheorem{cor}[thm]{Corollary}
\newtheorem{lem}[thm]{Lemma}

\newtheorem{con}[thm]{Conjecture}

\theoremstyle{definition}

\theoremstyle{remark}
\newtheorem{rem}[thm]{Remark}

\DeclareMathOperator{\erf}{erf} \DeclareMathOperator{\Id}{Id}
\DeclareMathOperator{\tr}{tr} \DeclareMathOperator{\ad}{ad}

\newcommand{\1}{\mathbf{1}}
\newcommand{\A}{{\mathcal{A}}}
\newcommand{\AR}{{\mathcal{A}(\R)}}
\newcommand{\AG}{{\mathcal{A}(G)}}
\newcommand{\C}{\mathbb{C}}
\newcommand{\Ceps}{{\mathcal{C}_\eps}}

\newcommand{\CM}{\mathcal{C}}

\newcommand{\D}{\Delta}
\newcommand{\SD}{{\sqrt{\Delta}}}
\newcommand{\eps}{\varepsilon}
\newcommand{\EW}{{E^\omega}}
\newcommand{\FCT}{{\mathcal{F}_{c,\vartheta}}}
\newcommand{\FUn}{{\mathcal{F}_{U,n}}}
\newcommand{\FVn}{{\mathcal{F}_{V,n}}}
\newcommand{\func}{\varphi}
\newcommand{\g}{\mathfrak{g}}
\newcommand{\ka}{{\kappa^\eps_\alpha}}
\newcommand{\kb}{{\kappa^\eps_\beta}}
\newcommand{\M}{{\mathcal{M}}}
\newcommand{\metric}{\mathbf{g}}

\newcommand{\N}{\mathbb{N}}
\newcommand{\OC}{{\mathcal{O}(\C)}}
\newcommand{\orb}{\gamma}

\newcommand{\R}{\mathbb{R}}
\newcommand{\Rn}{{\mathbb{R}^n}}
\newcommand{\RG}{{\mathcal{R}(G)}}
\newcommand{\U}{\mathcal{U}}
\newcommand{\VG}{{\mathcal{V}(G)}}
\newcommand{\WNT}{{\mathcal{W}_{N,\vartheta}}}

\begin{document}
\title{Analytic factorization of Lie group representations}
\author{Heiko Gimperlein, Bernhard Kr\"{o}tz, Christoph Lienau}
\date{}

\maketitle
\begin{abstract}
\noindent For every moderate growth representation $(\pi,E)$  of a real Lie
group $G$ on a Fr\'{e}chet space, we prove a factorization theorem of Dixmier--Malliavin
type for the space of analytic vectors $E^{\omega}$. There exists a
natural algebra of superexponentially decreasing analytic functions
$\AG$, such that $\EW=\Pi(\AG) \ \EW$. As a corollary we obtain that
$\EW$ coincides with the space of analytic vectors for the Laplace--Beltrami operator on
$G$.
\end{abstract}

\section{Introduction}

Consider a category $\CM$ of modules over a nonunital algebra $\A$. 
We say that $\CM$ has the \emph{factorization
property} if for all $\M \in \CM$,
$$\M \ =\ \A \cdot \M \ :=\ \mathrm{span}\left\{a \cdot m \ \mid\ a \in \A,\ m \in \M\right\}.$$
In particular, if $\A\in \CM$ this implies $\A = \A\cdot \A$.

Let $(\pi, E)$ be a representation of a real Lie group $G$ on a
Fr\'{e}chet space $E$. Then the corresponding space of smooth
vectors $E^\infty$ is again a Fr\'{e}chet space. The representation
$(\pi, E)$ induces a continuous action $\Pi$ of the algebra
$C^{\infty}_{c}(G)$ of test functions on $E$ given by
\[\Pi(f)v=\int_{G}f(g)\pi(g)v\ dg \ \ \ (f\in\mathrm{C}^{\infty}_{c}(G),v \in
E),\] which restricts to a continuous action on $E^\infty$. Hence
the smooth vectors associated to such representations are a
$C^{\infty}_{c}(G)$--module, and a result by Dixmier and Malliavin
\cite{dixmal} states that this category has the factorization
property.

In this article we prove an analogous result for the category of
analytic vectors.

For simplicity, we outline our approach for a Banach representation
$(\pi, E)$. In this case, the space $\EW$ of analytic vectors is
endowed with a natural inductive limit topology, and gives rise to a
representation $(\pi, \EW)$. To define an
appropriate algebra acting on $\EW$, we fix a left--invariant
Riemannian metric on $G$ and let $d$ be the associated distance
function. The continuous functions $\RG$ on $G$ which decay faster
than $e^{-n d(g,\1)}$ for all $n \in \N$ form a $G \times G$--module
under the left--right regular representation. We define $\AG$ to be
the space of analytic vectors of this action. Both $\RG$ and $\AG$
form an algebra under convolution, and the action $\Pi$ of
$C_c^\infty(G)$ extends to give $\EW$ the structure of an
$\AG$--module.

In this setting, our main theorem says that the category of analytic
vectors for Banach representations of $G$ has the factorization
property. More generally, we obtain a result for F--representations:
\begin{thm}\label{MainTheorem}
Let $G$ be a real Lie group and $(\pi,E)$ an $F$--representation of
$G$. Then
$$\AG \ =\ \AG \ast \AG$$
and
$$\EW\ =\  \Pi(\AG) \ \EW\ = \ \Pi(\AG) \ E.$$
\end{thm}
Let us remark that the special case of bounded Banach representations
of $(\mathbb{R},+)$ has been proved by one of the authors in \cite{lienau}. \\
As a corollary of Theorem \ref{MainTheorem} we obtain that a vector is analytic if and only if it
is analytic for the Laplace--Beltrami operator, which generalizes a result of Goodman \cite{goodman2}
for unitary representations. \\
In particular, the theorem extends Nelson's result that
$\Pi(\AG) \ \EW$ is dense in $\EW$ \cite{nelson}. G\r{a}rding had
obtained an analogous theorem for the smooth vectors \cite{garding}.
However, while Nelson's proof is based on approximate units constructed from
the fundamental solution $\varrho_t \in \AG$ of the heat equation on
$G$ by letting $t \to 0^+$, our strategy
relies on some more sophisticated functions of the Laplacian.\\
To prove Theorem \ref{MainTheorem}, we first consider the case
$G=(\mathbb{R},+)$. Here the proof is based on the key identity
\[\alpha_\eps(z) \cosh(\eps z) + \beta_\eps(z) = 1,
\]
for the entire functions $\alpha_\eps(z) =2 e^{-\eps z \erf(z)}$ and
$\beta_\eps(z) = 1-\alpha_\eps(z) \cosh(\eps z)$ on the complex
plane~\footnote{Some basic properties of these functions and the
Gaussian error function $\erf$ are collected in the appendix.}. We
consider this as an identity for the symbols of the Fourier
multiplication operators $\alpha_\eps(i\partial)$,
$\beta_\eps(i\partial)$ and $\cosh(i\eps\partial)$. The functions
$\alpha_\eps$ and $\beta_\eps$ are easily seen to belong to the
Fourier image of $\AR$, so that $\alpha_\eps(i\partial)$ and
$\beta_\eps(i\partial)$ are given by convolution with
some $\ka, \kb \in \AR$. For every $v \in \EW$ and sufficiently
small $\eps >0$, we may also apply $\cosh(i\eps\partial)$ to the
orbit map $\orb_v(g) = \pi(g)v$ and conclude that
$$(\cosh(i\eps \partial)\ \orb_v) \ast \ka + \orb_v \ast \kb = \orb_v.$$
The theorem follows by evaluating in $0$. \\
Unlike in the work of Dixmier and Malliavin, the rigid nature of
analytic functions requires a global geometric approach in the
general case. The idea is to refine the functional calculus of
Cheeger, Gromov and Taylor \cite{cgt} for the Laplace-Beltrami
operator in the special case of a Lie group. Using this tool, the
general proof then closely mirrors the argument for $(\R,+)$.\\
The article concludes by showing in Section \ref{related} how our
strategy may be adapted to solve some related factorization
problems.

\section{Basic Notions of Representations}

For a Hausdorff, locally convex and sequentially complete
topological vector space $E$ we denote by $GL(E)$ the associated
group of isomorphisms. Let $G$ be a Lie group. By a {\it
representation} $(\pi, E)$ of $G$ we understand a group homomorphism
$\pi: G \to GL(E)$ such that the resulting action
$$ G\times E \to E, \ \ (g, v)\mapsto \pi(g)v,$$
is continuous. For a vector $v\in E$ we shall denote by
$$\gamma_v: G\to E, \ \ g\mapsto \pi(g)v,$$
the corresponding continuous orbit map.

\par If $E$ is a Banach space, then
$(\pi, E)$ is called a {\it Banach representation}.

\begin{rem} \label{B-rep} Let $(\pi, E)$ be a Banach representation.
The uniform boundedness principle implies that
the function
$$w_\pi: G \to \R_+, \ \ g\mapsto \|\pi(g)\|,$$
is a {\it weight}, i.e. a locally bounded submultiplicative
positive function on $G$. \end{rem}

A representation $(\pi, E)$ is called an
{\it F-representation} if

\begin{itemize}
\item $E$ is a Fr\'echet space.
\item There exists a countable family of seminorms $(p_n)_{n\in \N}$
which define the topology of $E$ such that for every $n\in \N$
the action $G \times (E,p_n)\to (E,p_n)$ is continuous. Here $(E, p_n)$
stands for the vector space $E$ endowed with the topology induced
from $p_n$.
\end{itemize}

\begin{rem} (a) Every Banach representation is an $F$-representation.
\par (b) Let $(\pi, E)$ be a Banach representation and $\{X_n : n \in \N\}$
a basis of the universal enveloping algebra $\U(\g)$ of the Lie
algebra of $G$. Define a topology on the space of smooth vectors
$E^\infty$ by the seminorms $p_n(v) = \|d\pi(X_n) v\|$. Then the
representation $(\pi, E^\infty)$ induced by $\pi$ on this subspace
is an $F$-representation (cf.~\cite{bk}).
\par (c) Endow $E=C(G)$ with the topology of compact convergence. Then $E$
is a Fr\'echet space and $G$ acts continuously on $E$ via right
displacements in the argument. The corresponding representation
$(\pi, E)$, however, is not an $F$-representation.
\end{rem}

\subsection{Analytic vectors}
If $M$ is a complex manifold and $E$ is a topological vector space,
then we denote by $\mathcal{O}(M,E)$ the space of $E$-valued holomorphic maps.
We remark that $\mathcal{O}(M,E)$ is a topological vector space with regard
to the compact-open topology.

\par Let us denote by $\g$ the Lie algebra of $G$ and by $\g_\C$ its complexification.
We assume that $G\subset G_\C$ where $G_\C$ is a Lie group with Lie
algebra $\g_\C$. Let us stress that this assumption is superfluous
but simplifies notation and exposition. We denote by $\U_\C$ the set
of open neighborhoods of $\1 \in G_\C$.

\par If $(\pi,E)$ is a representation, then we call a vector $v\in E$ {\it analytic}
if the orbit map $\gamma_v: G\to E$ extends to a holomorphic map to
some $GU$ for $U\in \U_\C$. The space of all analytic vectors is
denoted by $E^\omega$. We note the natural embedding
$$E^\omega \to \lim_{U \to \{\1\}} \mathcal{O}(G U , E), \ \ v\mapsto \gamma_v,$$
and topologize $E^\omega$ accordingly.

\section{Algebras of superexponentially decaying functions}

We wish to exhibit natural algebras of functions acting on
$F$-representations.
For that let us fix a left invariant Riemannian metric $\metric$ on $G$.
The corresponding Riemannian measure $dg$ is a left
invariant Haar measure on $G$.
We denote by $d(g,h)$ the distance function associated to $\metric$
(i.e. the infimum of the lengths of all paths connecting $g$ and $h$)
and set
$$ d(g) :=d(g,\1) \qquad (g\in G) \, .$$
Here are two key properties of $d(g)$, see \cite{garding}:

\begin{lem}\label{lem=sm} If
$w : G \to \R_+$ is locally bounded and submultiplicative (i.e. $w(gh) \leq w(g)w(h)$), then there exist
$c,C>0$ such that
$$w(g) \leq C e^{c d(g)} \qquad (g\in G).$$
\end{lem}

\begin{lem} There exists $c>0$ such that for all $C>c$, $\int e^{-C d(g)} \ dg < \infty$.
\end{lem}

We introduce the space of {\it superexponentially decaying
continuous functions} on $G$ by
$$\RG := \left\{ \func \in
C(G) \mid \forall n\in \N  : \sup_{g \in G} |\func(g)|\ e^{n d(g)} <
\infty\right\}.$$ It is clear that $\RG$ is a Fr\'echet space which
is independent of the particular choice of the metric $\metric$. A
simple computation shows that $\RG$ becomes a Fr\'echet algebra
under convolution
$$\func*\psi(g)=\int_G \func(x)\ \psi(x^{-1}g) \ dx \qquad (\func, \psi \in \RG, g\in G)\, .$$
We remark that the left-right regular representation $L\otimes R$
 of $G\times G$ on $\RG$ is an $F$-representation.

\par If $(\pi, E)$ is an $F$-representation, then Lemma \ref{lem=sm} and Remark \ref{B-rep} imply
that
$$\Pi(\func)v:= \int_G \func(g)\ \pi(g)v\ dg \qquad (\func\in \RG, v\in E)$$
defines an absolutely convergent integral. Hence the prescription
$$\RG\times  E\to E, \ \ (\func, v)\mapsto \Pi(\func)v,$$
defines a continuous algebra action  of $\RG$ (here continuous refers to the continuity of the
bilinear map  $\RG\times E\to E$).

\par Our concern is now with the analytic vectors
of $(L\otimes R, \RG)$. We set $\AG:=\RG^\omega$
and record that
 $$\AG = \lim_{U \to \{\bf 1\}} \RG_U,$$
where $$\RG_U = \left\{\func \in \mathcal{O}(UGU) \mid  \forall Q \Subset
U \ \forall n\in \N:   \sup_{g \in G} \sup_{q_1, q_2 \in Q}
\left|\func(q_1 g q_2)\right|\ e^{n d(g)} < \infty \right\}.$$ It is clear that $\AG$
is a subalgebra of $\RG$ and that
$$\Pi(\AG)\ E \subset \EW$$
whenever $(\pi, E)$ is an $F$-representation.

\section{Some geometric analysis on Lie groups}\label{SectionGeometricAnalysis}

Let us denote by $\VG$ the space of left-invariant vector fields on $G$.
It is common to identify $\g$ with $\VG$ where $X\in \g$ corresponds to the
vector field $\widetilde X$ given by

$$(\widetilde{X} f) (g) = \frac{d}{dt}\Big|_{t=0} f(g \exp(t X))\ \qquad (g\in G, f\in C^\infty(G))\, .$$
We note that the adjoint of $\widetilde X$ on the Hilbert space
$L^2(G)$ is given by
$$\widetilde{X}^*= -\widetilde{X}  - \tr (\ad X)\, .$$
Note that  $\widetilde X^*=-\widetilde X$ in case $\g$ is unimodular. Let us
fix an an orthonormal basis $X_1, \dots, X_n$ of $\g$ with respect
to $\metric$. Then the Laplace--Beltrami operator $\Delta=d^*d$
associated to $\metric$ is given explicitly by
$$\D = \sum_{j=1}^n (-\widetilde{X_{j}} - \tr (\ad X_j))\ \widetilde{X_{j}}\, .$$
As $(G, \metric)$ is complete, $\D$ is essentially selfadjoint. We
denote by
$$\SD = \int \lambda \ dP(\lambda)$$
the corresponding spectral resolution. It provides us with a
measurable functional calculus, which allows to define
$$f(\SD)= \int f(\lambda)\ dP(\lambda)$$
as an unbounded operator $f(\SD)$ on $L^2(G)$ with domain
$$D(f(\SD))=\left\{\func \in L^2(G) \mid \int |f(\lambda)|^2\ d\langle P(\lambda)\func, \func\rangle < \infty \right\}.$$
Let $c,\vartheta>0$. We are going to apply the above calculus to
functions in the space
\begin{eqnarray*}
\FCT =  \left\{\func \in \OC \mid \forall N \in \N: \sup_{z \in
\WNT}
|\func(z)|\ e^{c|z|} < \infty\right\},\\
\WNT = \left\{z \in \C \mid |\operatorname{Im} z| < N \right\} \cup \left\{z \in
\C \mid |\Im z| < \vartheta |\operatorname{Re} z| \right\}.
\end{eqnarray*}
The resulting operators are bounded on $L^2(G)$ and given by a
symmetric and left invariant integral kernel $K_f \in C^\infty(G
\times G)$. Hence there exists a convolution kernel $\kappa_f\in
C^\infty(G)$ with $\kappa_{f}(x)=\kappa_{f}(x^{-1})$ such that
$K_f(x,y) = \kappa_f(x^{-1}y)$, and for all $x \in G$:
$$f(\SD)\ \func= \int_G K_f(x,y)\ \func(y)\ dy = \int_G \kappa_f(y^{-1}x)\ \func(y) \ dy =
(\func \ast \kappa_f)(x).$$
A theorem by Cheeger, Gromov and Taylor
\cite{cgt} describes the global behavior:

\begin{thm}\label{KernelDecay}
Let $c, \vartheta>0$ and $f \in \FCT$ even. Then $\kappa_f \in \RG$.
\end{thm}

We are going to need an analytic variant of their result.

\begin{thm}\label{AnalyticKernelDecay}
Under the assumptions of the previous theorem: $\kappa_f \in \AG$.
\end{thm}
\begin{proof}
We only have to establish local regularity, as the decay
at infinity is already contained in \cite{cgt}.\\
The Fourier inversion formula allows to express $\kappa_f$ as an
integral of the wave kernel:
$$\kappa_f(\cdot) = K_f(\cdot,\1) = f(\SD)\ \delta_\1 = \int_\R \hat{f}(\lambda)\ \cos(\lambda \SD)\ \delta_\1 \ d\lambda.$$
As we would like to employ $\|\cos(\lambda
\SD)\|_{\mathcal{L}(L^2(G))}\leq 1$, we cut off a fundamental
solution of $\D^{k}$ to write
$$\delta_\1 = \D^k \varphi + \psi$$ for a fixed $k>\frac{1}{4}\operatorname{dim}(G)$ and
some compactly supported $\varphi, \psi \in L^2$. Hence,
$$\Delta^l \kappa_f(\cdot) = \int_\R \hat{f}^{(2k+2l)}(\lambda)\ \cos(\lambda \SD)\ \varphi\ d\lambda +\int_\R \hat{f}^{(2l)}(\lambda)\ \cos(\lambda \SD)\ \psi\
d\lambda.$$ In the appendix we show the following inequality for all $n\in \N$ and some constants $C_{n}, R>0$
$$|\hat{f}^{(l)}(\lambda)| \leq
C_{n}\ l!\ R^l e^{-n |\lambda|}.$$
Using $\|\cos(\lambda
\SD)\|_{\mathcal{L}(L^2(G))}\leq 1$ and the Sobolev inequality, we
obtain
$$|\D^l \kappa_f(\cdot)| \leq C_{1}\ (2l)!\ S^{2l}$$
for some $S>0$. A classical result by Goodman \cite{warner} now
implies the right analyticity of $\kappa_f$, while left analyticity
follows from $\kappa_f(x) = \kappa_f(x^{-1})$. Browder's theorem
(Theorem 3.3.3 in \cite{kp}) then implies joint analyticity.
\end{proof}
\subsection{Regularized distance function}
In the last part of this section we are going to discuss a holomorphic regularization
of the distance function. Later on this will be used to construct certain holomorphic
replacements for cut-off functions. \\
Consider the time--$1$ heat kernel $\varrho :=
\kappa_{e^{-\lambda^2}}$ and define $\tilde{d}$ on $G$ by
$$\tilde{d}(g) := e^{-\D} d(g) = \int_G \varrho(x^{-1}g)\ d(x) \ dx.$$
\begin{lem}\label{dReg}
There exist $U \in \U_\C$ and a constant $C_{U}>0$ such that
$\tilde{d} \in \mathcal{O}(GU)$ and for all $g \in G$ and all $u \in
U$
$$|\tilde{d}(gu) - d(g)| \leq C_U.$$
\end{lem}
\begin{proof}
According to Theorem \ref{AnalyticKernelDecay} the heat kernel $\varrho$ admits
an analytic continuation to a superexponentially decreasing function on $GU$ for some bounded $U\in \U_{\C}$.
This allows to extend $\tilde{d}$ to $GU$. To
prove the inequality, we consider the integral
$$\bar{\varrho}(y) = \int_G \varrho(x^{-1}y) \ dx$$
as a holomorphic function of $y \in GU$. By the left invariance of
the Haar measure and the normalization of the heat kernel,
$\bar{\varrho}=1$ on $G$, and hence on $GU$. Recall the triangle inequality on $G$: $|d(x)-d(g)|\leq
d(x^{-1}g)$. This implies the uniform bound
\begin{eqnarray*}
\left|\tilde d(gu)-d(g)\right| &=& \left|\int_G \varrho(x^{-1}gu)\ (d(x)-d(g)) \
dx\right|\\
&\leq& \int_G \left|\varrho(x^{-1}gu)\right|\ d(x^{-1}g) \ dx\\
&\leq&\sup_{v \in U} \int_G \left|\varrho(x^{-1}v)\right|\ d(x^{-1}) \ dx.
\end{eqnarray*}
\end{proof}


\section{Proof of the Factorization Theorem}

Let $(\pi,E)$ be a representation of $G$ on a sequentially complete
locally convex Hausdorff space and consider the Laplacian as an
element $$\D = \sum_{j=1}^n (-{X_{j}} - \tr (\ad X_j))\ {X_{j}}$$ of
the universal enveloping algebra of $\g$. A vector $v\in E$ will be
called {\it $\D$-analytic}, if there exists $\eps>0$ such that for
all continuous seminorms $p$ on $E$ one has

$$\sum_{j=0}^\infty \frac{\eps^{j}}{(2j)!}\ p(\Delta^j v)<\infty\, .$$

\begin{lem} \label{CauchyInEq} Let $E$ be a sequentially complete locally convex Hausdorff space
and $\func\in \mathcal{O}(U, E)$ for some $U\in \U_\C$. Then there
exists $R=R(U)>0$ such that for all continuous semi-norms $p$ on $E$
there exists a constant $C_p$ such that
$$p\left(\left(\widetilde{X_{i_1}} \cdots \widetilde{X_{i_k}} \func\right)\left(\1\right)\right) \leq C_p\ k!\ R^k$$
for all $(i_1, \ldots, i_k)\in \N^k$, $k\in \N$.
\end{lem}
\begin{proof}
There exists a small neighborhood of $0$ in $\g$ in which the mapping
\[\Phi: \g\rightarrow E, \ X\mapsto \func(\exp(X)),\]
is analytic. Let $X=t_{1}X_{1}+\cdots +t_{n}X_{n}$. Because $E$ is
sequentially complete, $\Phi$ can be written for small $X$ and $t$
as
\[\Phi(X)=\sum_{k=0}^{\infty}\frac{1}{k!}\mathop{\sum_{\alpha\in \mathbb{N}^{n}}}_{|\alpha|=k}\left(\widetilde{X_{\alpha_{1}}} \cdots \widetilde{X_{\alpha_{k}}} \func\right)\left(\1\right)t^{\alpha}.
\]
As this series is absolutely summable, there exists a $R>0$ such
that for every continuous semi-norm $p$ on $E$ there is a constant
$C_{p}$ with
$$p\left(\left(\widetilde{X_{i_1}} \cdots \widetilde{X_{i_k}} \func\right)\left(\1\right)\right) \leq C_p\ k!\ R^k$$
for all $(i_1, \ldots, i_k)\in \N^k$, $k\in \N$.
\end{proof}
As a consequence we obtain:

\begin{lem}\label{DeltaAnalytic} Let $(\pi, E)$ be a representation of $G$ on some sequentially complete
locally convex Hausdorff space $E$. Then analytic vectors are $\D$--analytic.
\end{lem}

In Corollary \ref{Danalytic} we will see that the converse holds for $F$-representations.

Let $(\pi, E)$ be an $F$-representation of $G$. Then for each $n\in \N$ there exists
$c_n, C_n>0$ such that

$$ \|\pi(g)\|_n \leq C_n \cdot e^{c_n d(g)} \qquad (g\in G),$$
where

$$\|\pi(g)\|_n:=\sup_{p_n(v)\leq 1\atop v\in E} p_n(\pi(g)v)\, .$$
For $U\in \U_\C$ and $n\in \N$ we set

$$\FUn  = \left\{\func \in \mathcal{O}(GU, E) \mid \forall Q \Subset U\ \forall \eps>0: \
\sup_{g \in G} \sup_{q \in Q} p_n(\func(g q))\ e^{-(c_n+\eps)  d(g)}
< \infty\right\}.$$

We are also going to need the subspace of superexponentially
decaying functions in $\bigcap_n \FUn$:
$$\mathcal{R}(GU,E) = \left\{\func \in \mathcal{O}(GU, E) \mid  \forall Q \Subset U \ \forall n,N \in \N: \
\sup_{g \in G} \sup_{q \in Q} p_n(\func(g q))\ e^{N  d(g)} <
\infty\right\}.$$

We record:
\begin{lem}\label{conv} If $\kappa \in \AG_V$, then right convolution with $\kappa$ is a bounded operator
from $\FUn$ to $\FVn$ for all $n\in \N$.
\end{lem}
We denote by $\Ceps$ the power series expansion $\sum_{j=0}^\infty
\frac{\eps^{2j}}{(2j)!} \D^j$ of $\cosh(\eps \SD)$. Note the
following consequence of Lemma \ref{CauchyInEq}:

\begin{lem}\label{ceps} Let $U, V\in \U_\C$ such that $V\Subset U$. Then there exists $\eps>0$ such that
$\Ceps$ is a bounded operator from
$\FUn$ to $\FVn$ for all $n\in\N$.
\end{lem}
As in the Appendix, consider the functions $\alpha_\eps(z) =2 e^{-\eps z \erf(z)}$ and $\beta_\eps(z) =
1-\alpha_\eps(z) \cosh(\eps z)$, which belong to the space $\mathcal{F}_{2\eps,\vartheta}$.
We would like to substitute $\SD$ into our key identity
(\ref{KeyId})
$$\alpha_\eps(z) \cosh(\eps z) + \beta_\eps(z) = 1$$
and replace the hyperbolic cosine by its Taylor expansion.

\begin{lem} \label{FunctionalCalculus} Let $U\in \U_\C$. Then there exist
$\eps>0$ and $V\subset U$ such that for any $\func \in \FUn$, $n\in
\N$,
$$\Ceps(\func) \ast \ka + \func \ast \kb = \func$$
holds as functions on $GV$.
\end{lem}
\begin{proof}
Note that $\ka, \kb \in \AG$ according to Theorem
\ref{AnalyticKernelDecay}. We first consider the case $E=\C$ and
 $\func \in L^2(G)$. With
$|\alpha_\eps(z) \cosh(\eps z)|$ being bounded, $\cosh(\eps \SD)$
maps its domain into the domain of $\alpha_\eps(\SD)$, and the rules
of the functional calculus ensure
$$\func - \beta_\eps(\SD) \func = (\alpha_\eps(\cdot) \cosh(\eps \cdot))(\SD) \func= (\cosh(\eps \SD)\func) \ast \ka$$
in $L^2(G)$ for all $\func \in D(\cosh(\eps \SD))$. For such
$\func$, the partial sums of $\Ceps \func$ converge to $\cosh(\eps
\SD)\func$ in $L^2(G)$, and hence almost everywhere. Indeed,
\begin{eqnarray*}
&& \hspace{-0.9cm}\left\|\cosh(\eps \SD)\func
-\sum_{j=0}^N\frac{\eps^{2j}}{(2j)!} \D^j\func
\right\|^2_{L^2(G)}\\&=& \int \left\langle dP(\lambda)\left(
\cosh(\eps \SD)\func -\sum_{j=0}^N\frac{\eps^{2j}}{(2j)!}\ \D^j
\func\right),\cosh(\eps \SD)\func
-\sum_{k=0}^N\frac{\eps^{2k}}{(2k)!}\
\D^k \func \right\rangle\\
&=& \int \left(\cosh(\eps \lambda)
-\sum_{k=0}^N\frac{(\eps\lambda)^{2k}}{(2k)!}\right)^2\ \langle
dP(\lambda) \func, \func \rangle\\
&=& \sum_{j,k=N+1}^\infty \int \frac{(\eps\lambda)^{2j}}{(2j)!}\
\frac{(\eps\lambda)^{2k}}{(2k)!}\ \langle dP(\lambda) \func, \func
\rangle\ ,
\end{eqnarray*}
and the right hand side tends to $0$ for $N \to \infty$, because
$$\sum_{j,k=0}^\infty \int \frac{(\eps\lambda)^{2j}}{(2j)!}\
\frac{(\eps\lambda)^{2k}}{(2k)!}\ \langle dP(\lambda) \func,
\func\rangle = \int \cosh(\eps \lambda)^2 \langle dP(\lambda)
\func,\func\rangle < \infty\ .$$ In particular, given $\func \in
\mathcal{R}(GU,E)$ and $\lambda \in E'$, we obtain $\Ceps
\lambda(\func) = \cosh(\eps \SD) \lambda(\func)$ almost everywhere
and
$$\Ceps(\lambda(\func)) \ast \ka + \lambda(\func) \ast \kb =
\lambda(\func)$$ as analytic functions on $G$ for sufficiently small
$\eps>0$.

Since the above identity holds for all $\lambda \in E'$, we obtain
$$\Ceps(\func) \ast \ka + \func \ast \kb = \func$$
on any connected domain $GV$, $\1\in V \subset U$, on which the left hand
side is holomorphic.

Recall the regularized distance function
$\tilde{d}(g) = e^{-\D} d(g)$ from Lemma \ref{dReg}, and set
$\chi_\delta(g) := e^{-\delta \tilde{d}(g)^2}$ ($\delta>0$). Given
$\func \in \FUn$, $\chi_\delta \func \in \mathcal{R}(GU,E)$ and
$$\Ceps(\chi_\delta \func) \ast \ka + (\chi_\delta\func) \ast \kb =
\chi_\delta\func\, .$$ The limit $\chi_\delta\func \to \func$ in $\FUn$ as $\delta\to 0$
is easily verified. From Lemma \ref{conv} we also get
$(\chi_\delta\func) \ast \kb \to \func \ast \kb$ as $\delta \to 0$.
Finally Lemma \ref{conv} and Lemma \ref{ceps} imply
$$\Ceps(\chi_\delta \func)\ast \ka \to \ \Ceps(\func)\ast \ka
\quad (\delta \to 0).$$ The assertion follows.
\end{proof}

\begin{proof}[Proof of Theorem \ref{MainTheorem}]
Given $v\in E^{\omega}$, the orbit map $\orb_{v}$
belongs to $\bigcap_{n} \FUn$ for some $U\in \U_{\C}$. Applying
Lemma \ref{FunctionalCalculus} to the orbit map and evaluating at $\1$
we obtain the desired factorization
\[v=\gamma_{v}(\1)=\Pi(\ka)\left(\Ceps(\gamma_{v})(\1)\right)+\Pi(\kb)\left(\gamma_{v}(\1)\right).
\]
\end{proof}

Note the following generalization of a theorem by Goodman for
unitary representations \cite{goodman2, warner}.
\begin{cor}\label{Danalytic}
Let $(\pi, E)$ be an F-representation. Then every $\D$--analytic
vector is analytic.
\end{cor}
\begin{rem}
a) A further consequence of our Theorem \ref{MainTheorem} is a
simple proof of the fact that the space of analytic vectors for a
Banach representation is complete.\\
b) We can also substitute $\SD$ into Dixmier's and Malliavin's
presentation of the constant function $1$ on the real line
\cite{dixmal}. This invariant refinement of their argument shows
that the smooth vectors for a Fr\'echet representation are precisely
the vectors in the domain of $\Delta^k$ for all $k \in \N$.
\end{rem}

\section{Related Problems}\label{related}

We conclude this article with a discussion of how our techniques can
be modified to deal with a number of similar questions.\\
In the context of the introduction, given a nonunital algebra $\A$,
a category $\CM$ of $\A$--modules is said to have the \emph{strong
factorization property} if for all $\M \in \CM$,
$$\M = \{a m\ \mid\ a \in \A,\ m \in \M\}.$$

\subsection{A Strong Factorization of Test Functions}

Our methods may be applied to solve a related strong factorization
problem for test functions. On $\R^n$ the Fourier transform allows
to write a test function $\func \in C_c^\infty(\Rn)$ as the
convolution $\psi \ast \Psi$ of two Schwartz functions, and
\cite{rst} posed the natural problem whether one could demand $\psi,
\Psi \in \mathcal{R}(\Rn)$. We are going to prove this in a more
general setting.

\begin{thm}
For every real Lie group $G$ $$C_c^\infty(G) \subset \left\{\psi\ast\Psi \ \mid \ \psi,\Psi\in\RG \right\}.$$
\end{thm}

As above, we first regularize an appropriate distance function and
set
$$\l(z) =  \frac{1}{\sqrt{\pi}} e^{-z^2} \ast \log(1+|z|).$$
\begin{lem}
The function $\l(z)$ is entire and approximates $\log(1+|z|)$ in the
sense that for all $N>0$, $\vartheta \in (0,1)$ there exists a
constant $C_{N, \vartheta}$ such that
$$|\l(z) - \log(1+|z|)| \leq C_{N, \vartheta} \quad(z \in \WNT).$$
\end{lem}

Let $m \in \N$. We would like to substitute the square root of the
Laplacian associated to a left invariant metric $G$ into a
decomposition
$$1=\widehat \psi_m(z)\ \widehat \Psi_m(z)$$
of the identity. In the current situation we use $\widehat\psi_m(z)
= e^{-m \l(z)}$ and $\widehat \Psi_m(z) = e^{m\l(z)}$. Denote the
convolution kernels of $\widehat \psi_m(\SD)$ and $\widehat
\Psi_m(\SD)$ by $\psi_m$ resp.~$\Psi_m$. The ideas from the proof of
Theorem \ref{AnalyticKernelDecay} may be combined with the results
of \cite{cgt} to obtain:

\begin{lem}
Let $\chi \in C_c^\infty(G)$ with $\chi = 1$ in a neighborhood of
$\1$. Then $\chi \Psi_m$ is a compactly supported distribution of
order $m$ and $(1-\chi)\Psi_m \in \RG \cap C^\infty(G)$. Given $k
\in \N$, $\psi_m \in \RG \cap C^k(G)$ for sufficiently large $m$.
\end{lem}

Therefore $\widehat \Psi_m(\SD)$ maps $C_c^\infty(G)$ to $\RG$. The
functional calculus leads to a factorization
$$\Id_{C_c^\infty(G)} = \widehat \psi_m(\SD)\ \widehat \Psi_m(\SD)$$
of the identity, and in particular for any $\func \in
C_c^\infty(G)$,
$$\func = (\widehat \Psi_m(\SD)\ \func) \ast \psi_m \in \RG \ast \RG.$$

\subsection{Strong Factorization of $\AG$}

It might be possible to strengthen Theorem \ref{MainTheorem} by
showing that the analytic vectors have the strong factorization
property.
\begin{con}
For any F-representation $(\pi, E)$ of a real Lie group $G$,
$$\EW = \{\Pi(\func) v\ \mid \ \func \in \AG, \ v \in \EW\}.$$
\end{con}
We provide some evidence in support of this conjecture and verify it
for Banach representations of $(\R, +)$ using hyperfunction
techniques.
\begin{lem}
The conjecture holds for every Banach representation of $(\R, +)$.
\end{lem}
\begin{proof}Let $(\pi,E)$ be a representation of $\mathbb{R}$ on a Banach
space $(E,\left\|\cdot\right\|)$. Then there exist constants $c,C>0$
such that $\|\pi(x)\|\leq C e^{c|x|}$ for all $x\in \mathbb{R}$. If
$v\in E^{\omega}$, there exists $R>0$ such that the orbit map
$\gamma_{v}$ extends holomorphically to the strip
$S_{R}=\{z\in\mathbb{C} \ | \ \operatorname{Im}z\in (-R,R) \}$. Let
\begin{align*}
\mathcal{F}_{+}\left(\gamma_{v}\right)(z) &=\int_{-\infty}^{0}\gamma_{v}(t)e^{-itz}\ dt, \ \ \ \operatorname{Im}z>c, \\
-\mathcal{F}_{-}\left(\gamma_{v}\right)(z)
&=\int_{0}^{\infty}\gamma_{v}(t)e^{-itz}\ dt, \ \ \
\operatorname{Im}z<-c.
\end{align*}
Define the Fourier transform $\mathcal{F}(\gamma_{v})$ of $\gamma_{v}$ by
\[\mathcal{F}\left(\gamma_{v}\right)(x)=\mathcal{F}_{+}\left(\gamma_{v}\right)(x+2ic)-
\mathcal{F}_{-}\left(\gamma_{v}\right)(x-2ic).\] Note that
$\left\|\mathcal{F}\left(\orb_v\right)(x)\right\|e^{r\left|x\right|}$
is bounded for every $r<R$. Let $g(z) := \frac{Rz}{2} \erf(z)$ and
write $\mathcal{F}\left(\gamma_{v}\right)$ as
\begin{align}\label{fo}\mathcal{F}\left(\gamma_{v}\right)=e^{-g}e^{g}\mathcal{F}\left(\gamma_{v}\right)
\end{align}
Define the inverse Fourier transform $\mathcal{F}^{-1}(\mathcal{F}(\gamma_{v}))$
for $x\in \mathbb{R}$ by
\[\mathcal{F}^{-1}(\mathcal{F}(\gamma_{v}))(x)=\int_{\operatorname{Im}t=2c}\mathcal{F}_{+}\left(\gamma_{v}\right)(t)e^{itx} \ dt
-\int_{\operatorname{Im}t=-2c}\mathcal{F}_{-}\left(\gamma_{v}\right)(t)e^{itx}
\ dt.
\]
Applying the inverse Fourier transform to both sides of $(\ref{fo})$ and evaluating at $0$ yields
\begin{align*}
v=(2\pi)^{-1}
\Pi\left(\mathcal{F}^{-1}\left(e^{-g}\right)\right)\left(\mathcal{F}^{-1}\left(e^{g}\mathcal{F}(\gamma_{v})\right)(0)\right).
\end{align*}
The assertion follows because
$\mathcal{F}^{-1}\left(e^{-g}\right)\in \mathcal{A}(\mathbb{R})$.
\end{proof}

Strong factorization likewise holds for Banach representations of
$(\mathbb{R}^{n},+)$. Using the Iwasawa decomposition we are able to deduce from this
the conjecture for $SL_{2}(\mathbb{R})$. 

\begin{appendix}
\section{An Identity of Entire Functions}

Consider the following space of exponentially decaying holomorphic
functions
\begin{eqnarray*} \FCT =  \left\{\func \in \OC \mid
\forall N\in \N: \sup_{z \in \WNT}
|\func(z)|\ e^{c|z|} < \infty\right\},\\
\WNT = \left\{z \in \C \mid |\Im z| < N \right\} \cup \left\{z \in \C \mid |\Im z| <
\vartheta |\Re z| \right\}.
\end{eqnarray*}
To understand the convolution kernel of a Fourier multiplication
operator on $L^2(\R)$ with symbol in $\FCT$, or more generally
functions of $\SD$ on a manifold as in Section
\ref{SectionGeometricAnalysis}, we need some properties of the
Fourier transformed functions.

\begin{lem}
Given $f \in \FCT$, there exist $C, R>0$ such that
$$|\hat{f}^{(k)}(z)| \leq C_n\ k!\ R^k e^{-n|z|}$$
for all $k, n \in \N$.
\end{lem}
\begin{proof}
Given $f \in \FCT$, the Fourier transform extends to a
superexponentially decaying holomorphic function on
$\mathcal{W}_{c,\vartheta}$. It follows from Cauchy's integral
formula that
$$|\hat{f}^{(k)}(z)| \leq C_n\ k!\ R^k e^{-n|z|}$$ for all $k, n \in \N$.
\end{proof}

Some important examples of functions in $\FCT$ may be constructed
with the help of the Gaussian error function \cite{wolfram}
$$\erf(x) = \frac{2}{\sqrt{\pi}} \int_{-\infty}^x e^{-t^2}\ dt.$$
The error function extends to an odd entire function, and $\erf(z) - 1 =
O(z^{-1} e^{-z^2})$ as $z \to \infty$ in a sector $\{|\operatorname{Im} z| <
\vartheta \operatorname{Re} z\}$ around $\R_+$.
\begin{rem}
The function
$$z \erf(z) = \frac{1}{\sqrt{\pi}} e^{-z^2} \ast |z| - \frac{1}{\sqrt{\pi}}e^{-z^2}$$
is just one convenient regularization of the absolute value $|z|$,
and the basic properties we need also hold for other similarly
constructed functions. For example replace the heat kernel $ \frac{1}{\sqrt{\pi}}e^{-z^2}$
by a suitable analytic probability density.
\end{rem}

For any $\eps>0$, some algebra shows that the even entire functions
$\alpha_\eps(z) =2 e^{-\eps z \erf(z)}$ and $\beta_\eps(z) =
1-\alpha_\eps(z) \cosh(\eps z)$ decay exponentially as $z \to
\infty$ in $\WNT$ for any $\vartheta<1$. Hence $\alpha_\eps,
\beta_\eps \in \mathcal{F}_{2\eps,\vartheta}$. Our
later factorization hinges on a multiplicative decomposition of the
constant function $1$:
\begin{lem}\label{KeyId}
For all $\eps>0,\vartheta\in\left(0,1\right)$, the functions
$\alpha_\eps, \beta_\eps \in \mathcal{F}_{2\eps,\vartheta}$ satisfy
the identity
$$\alpha_\eps(z) \cosh(\eps z) + \beta_\eps(z) = 1.$$
\end{lem}

\end{appendix}

\end{document}